\documentclass[a4paper]{article}
\usepackage{a4wide}
\usepackage{amsmath,amssymb,amsthm}
\usepackage{enumerate}
\usepackage{abstract}
\usepackage{mathpazo}
\makeatletter\@ifpackageloaded{mathpazo}\@tempswatrue\@tempswafalse
\if@tempswa
  \DeclareFontFamily{OT1}{pzc}{}
  \DeclareFontShape{OT1}{pzc}{m}{it}{<-> s * [1.15] pzcmi7t}{}
  \DeclareMathAlphabet{\mathpzc}{OT1}{pzc}{m}{it}
\fi\makeatother
\usepackage{hyperref}
\usepackage{thmtools}
\usepackage[T1]{fontenc}
\usepackage[utf8]{inputenc}

\usepackage[style=alphabetic,citestyle=alphabetic,sorting=nty,firstinits=true,maxbibnames=4,backend=bibtex]{biblatex}
\makeatletter\@ifpackageloaded{biblatex}{%
  \bibliography{references}
  \renewbibmacro{in:}{%
    \ifentrytype{incollection}{\printtext{\bibstring{in}\intitlepunct}}{}}
  \renewbibmacro{publisher+location+date}{%
    \iflistundef{publisher}
      {\setunit*{\addcomma\space}}
      {\setunit*{\addcomma\space}}%
    \printlist{publisher}%
    \setunit*{\addcomma\space}%
    \printlist{location}%
    \setunit*{\addcomma\space}%
    \usebibmacro{date}%
    \newunit}
  \DeclareFieldFormat[article]{pages}{#1\isdot}
  \DeclareFieldFormat[article,incollection,inproceedings,unpublished]{title}{#1\isdot}
  \DeclareFieldFormat[article]{journaltitle}{\mkbibemph{#1\isdot}}
  
}{}\makeatother

\makeatletter\@ifpackageloaded{hyperref}{%
  \usepackage{xcolor}
  \definecolor{dark-red}{rgb}{0.4,0.15,0.15}
  \definecolor{dark-blue}{rgb}{0.15,0.15,0.4}
  \definecolor{medium-blue}{rgb}{0,0,0.5}
  \hypersetup{
    colorlinks,
    linkcolor={dark-red},
    citecolor={dark-blue},
    urlcolor={medium-blue}%
  }
}{}\makeatother

\newcommand{\address}[1]{\vspace{-2em}\begin{center}{\footnotesize #1}\end{center}}

\usepackage{sectsty}
\sectionfont{\large\bfseries}
\subsectionfont{\normalsize\bfseries}

\newcommand{\mirror}[1]{\widetilde{#1}}
\newcommand{\A}{\mathpzc{A}}
\newcommand{\B}{\mathpzc{B}}
\newcommand{\D}{\mathpzc{D}}
\newcommand{\R}{\mathcal{R}}
\newcommand{\PP}{\mathpzc{P}}
\newcommand{\M}{\mathpzc{M}}
\newcommand{\Fac}{\mathpzc{F} \!}
\newcommand{\Pal}{\mathpzc{Pal} \!}
\newcommand{\Pri}{\mathpzc{Pri}}
\newcommand{\del}{\partial}

\declaretheorem[numberwithin=section,refname={theorem,theorems},Refname={Theorem,Theorems}]{theorem}

\declaretheorem[sibling=theorem,name=Lemma]{lemma}
\declaretheorem[sibling=theorem,name=Proposition]{proposition}
\declaretheorem[sibling=theorem,name=Corollary]{corollary}

\newcommand{\keywords}[1]{\par\noindent{\footnotesize{\em Keywords\/}: #1}}

\begin{document}
  \title{Privileged Factors in the Thue-Morse Word -- A Comparison of Privileged Words and Palindromes}
  \author{Jarkko Peltomäki\\
          \small \href{mailto:jspelt@utu.fi}{jspelt@utu.fi}}
  \date{}
  \maketitle
  \address{Turku Centre for Computer Science TUCS, 20520 Turku, Finland\\
  University of Turku, Department of Mathematics and Statistics, 20014 Turku, Finland}

  \noindent
  \hrulefill
  \begin{abstract}
    \vspace{-1em}
    \noindent
    In this paper we study the privileged complexity function of the Thue-Morse word. We prove a recursive formula
    describing this function, and using the formula we show that the function is unbounded and that the values of
    the function have arbitrarily large gaps of zeros. This demonstrates that the privileged complexity function of an
    infinite word can drastically differ from its palindromic complexity function, even though there
    are relations between these functions. Further we study the behavior of palindromes and privileged words in
    infinite words and the relation between rich words and privileged words.
    \vspace{1em}
    \keywords{thue-morse word, palindrome, privileged word, return word, rich word}
    \vspace{-1em}
  \end{abstract}
  \hrulefill

  % footnote ilman numeroa
  \let\thefootnote\relax\footnotetext{{\scriptsize \textcopyright\, 2015, Elsevier. Licensed under the
  \href{http://creativecommons.org/licenses/by-nc-nd/4.0/}%
  {Creative Commons Attribution-NonCommercial-NoDerivatives 4.0 International}.}}

  \section{Introduction}
  In this paper we continue the study of so-called privileged words. The study of privileged words was initiated in the
  articles \cite{2013:a_characterization_of_subshifts_with_bounded_powers} and
  \cite{2013:introducing_privileged_words_privileged}. In
  \cite{2013:a_characterization_of_subshifts_with_bounded_powers} privileged words were used as a technical tool in
  order to characterize aperiodic and minimal subshifts with bounded powers. This new class of words has also interest
  in its own right. For example the paper \cite{2013:introducing_privileged_words_privileged} characterized Sturmian
  words using their privileged complexity function (the privileged complexity function of a word counts the number of
  factors of given length that are privileged words).

  Privileged words also appear in \cite{2013:on_palindromic_factorization_of_words}. The authors of this paper consider
  the palindromic length of a word, that is, the minimal number of palindromes needed to write a given word as a
  concatenation of palindromes. The authors observe that some of their arguments need only properties shared by both
  privileged words and palindromes, so results on the analogous privileged length of a word are obtained. Namely, it is
  proven that given an infinite $k$-power free word $w$, for every positive integer $P$ there exists a prefix of $w$
  which cannot be expressed as a concatenation of at most $P$ privileged words.

  The motivation for defining privileged words comes from the theory of rich words \cite{2009:palindromic_richness}.
  Rich words are words containing maximal number of palindromes as factors, and they are characterized by the fact
  that in a rich word every palindrome is a complete first return to a shorter palindrome. By altering this condition
  slightly, we get the definition of privileged words: a word is privileged if it is a complete first return to a
  shorter privileged word; the shortest privileged words are the letters of the alphabet and the empty word. At first
  glance, privileged words have nothing to do with palindromes. In general, this is indeed true, but there are some
  connections to rich words. Namely in
  \cite{2013:a_characterization_of_subshifts_with_bounded_powers} and
  \cite{2013:introducing_privileged_words_privileged} it was observed that a word is rich if and only if its set of
  palindromes coincides with its set of privileged words. Every position in a word introduces exactly one new
  privileged factor and at most one new palindromic factor. This means that in a sense privileged factors which are not
  palindromes measure the palindromic defect of a word. The palindromic defect of a word is the number of positions
  in the word that do not introduce a new palindrome; rich words are the words with zero defect.
  Other than this connection to richness, privileged words have other similarities with palindromes. For instance a
  privileged prefix of a privileged word occurs also as a suffix, just as a palindromic prefix of a palindrome occurs
  also as a suffix.

  In general privileged words behave differently from palindromes. In this paper we show that the privileged
  complexity function of an infinite word may behave in a much more complex way than its palindromic complexity
  function. We investigate the privileged complexity function $\A_t$ of the Thue-Morse word $t$ and prove that $\A_t$
  is unbounded and that its values have arbitrarily long gaps of zeros. In contrast it has been proven in
  \cite{2000:palindrome_complexity_bounds_for_primitive_substitution} and \cite{2003:palindrome_complexity} that the
  palindromic complexity function of fixed points of primitive morphisms is bounded. Also, as is easily seen, for any
  palindromic complexity function two consecutive zeros would imply that the function takes only finitely many nonzero
  values. To obtain the results we derive complete recursive formulas for computing $\A_t$. Various other kinds of
  complexity functions for the Thue-Morse word have been considered in the past. Recently in
  \cite{2013:on_the_number_of_unbordered_factors} Goč, Mousavi and Shallit proved that there exists a system of
  recurrences for the number of unbordered factors of given length in the Thue-Morse word.

  After defining the necessary notation and definitions, in \autoref{sec:thue-morse} we conclude the study of the
  function $\A_t$ initiated in \cite{2013:introducing_privileged_words_privileged}, and present the main results of
  this paper. We describe the relations between different types of privileged factors in the Thue-Morse word, and give
  a recursive formula for computing $\A_t$. We study the asymptotic behavior of this function, and prove that it is
  unbounded. We also verify a conjecture of \cite{2013:introducing_privileged_words_privileged} stating that the values
  of the function $\A_t$ contain arbitrarily long gaps of zeros.

  In \autoref{sec:privileged_palindrome} we briefly study the privileged palindrome complexity function of the
  Thue-Morse word, that is, the function counting the number of privileged factors which are also palindromes. As in
  \autoref{sec:thue-morse}, we give a recursive formula for computing the values of this function, and study its
  asymptotic behavior and gaps of zeros.

  In the last section we compare palindromes and privileged words. First we further sharpen the connection to
  palindromic richness mentioned above. Secondly we investigate when palindromic factors can form a proper subset of
  privileged factors in a word (remember: in rich words palindromes are exactly the privileged words).

  \section{Preliminaries}
  In this text, we let $A$ denote a finite \emph{alphabet}, which is a finite non-empty set of symbols. The elements
  of $A$ are called \emph{letters}. A (finite) \emph{word} over $A$ is a sequence of letters. To the empty sequence
  corresponds the \emph{empty word}, denoted by $\varepsilon$. The set of all finite words over $A$ is denoted by $A^*$.
  The set of non-empty words over $A$ is the set $A^+ := A^* \setminus \{\varepsilon\}$. A natural operation on words
  is concatenation. Under this operation $A^*$ is a free monoid over $A$. The concatenation of two words $u$ and $v$ is
  denoted by $uv$. The letters occurring in the word $w$ form the \emph{alphabet of $w$}. From now on we assume that
  binary words are over the alphabet $\{0,1\}$. Given a finite word $w = a_1 a_2 \cdots a_n$ of $n$ letters, we say
  that the \emph{length} of $w$, denoted by $|w|$, is equal to $n$. By convention the length of the empty word is $0$.
  The set of all words of length $n$ over the alphabet $A$ is denoted by $A^n$. A word $w$ is \emph{primitive} if $w$
  is of the form $z^n$ if and only if $n = 1$. A word is \emph{non-primitive} if it is not primitive.

  An \emph{infinite word} $w$ over $A$ is a function from the natural numbers to $A$. We consider such a function as a
  sequence indexed by the natural numbers with values in $A$. We write concisely $w = a_1 a_2 a_3 \cdots$ with $a_i \in A$.
  The set of infinite words is denoted by $A^\omega$. For an infinite word $w$ we write $|w| = \infty$. An infinite
  word $w$ is said to be \emph{ultimately periodic} if we can write it in the form $w = uv^\omega = uvvv \cdots$ for
  some words $u,v \in A^*$. If $u = \varepsilon$, then $w$ is said to be \emph{periodic}. An infinite word which is not
  ultimately periodic is \emph{aperiodic}.

  A finite word $u$ is a \emph{factor} of the finite or infinite word $w$ if we can write that $w = zuv$ for some
  $z \in A^*$ and $v \in A^* \cup A^\omega$. If $z = \varepsilon$, then the factor $u$ is called a \emph{prefix} of
  $w$. If $v = \varepsilon$, then we say that $u$ is a \emph{suffix} of $w$. If a word $u$ is both a prefix and a suffix
  of $w$, then $u$ is a \emph{border} of $w$. The set of factors of $w$ is denoted by $\Fac(w)$. The set $\Fac_w(n)$
  is defined to contain all factors of $w$ of length $n$. We call $u$ \emph{a central factor} of $w$ if there exists a
  factorization $w = xuy$ with $|x| = |y|$. If $w = a_1 a_2 \cdots a_n$, then we let $w[i,j] = a_i \cdots a_j$
  whenever the choices of positions $i$ and $j$ make sense. This notion is extended to infinite words in a natural
  way. An \emph{occurrence} of $u$ in $w$ is a a position $i$ such that $w[i,i+|u|-1] = u$. If such a position exists,
  then we say that $u$ \emph{occurs} in $w$. By $|w|_u$ we denote the number of occurrences of the factor $u$ in the
  word $w$. If $|w|_u = 1$, then we say that $u$ is \emph{unioccurrent} in $w$. We say that a position $i$
  \emph{introduces} a factor $u$ if $w[i - |u| + 1, i] = u$ and $u$ is unioccurrent in $w[1,i]$. A
  \emph{complete first return} to the word $u$ is a word starting and ending with $u$ and containing exactly two
  occurrences of $u$ (occurrences of $u$ may overlap). A word which is a complete first return to some word is called
  a \emph{complete return word}. A \emph{complete return factor} is a factor of some word which is a complete return
  word.

  An infinite word $w$ is \emph{recurrent} if each of its factors occurs in it infinitely often. The word $w$ is
  called \emph{uniformly recurrent} if each factor $u$ occurs in it infinitely often, and the gap between two
  consecutive occurrences of $u$ in $w$ is bounded by a constant depending only on $u$. Equivalently $w$ is uniformly
  recurrent if for each factor $u$ there exists an integer $R$ such that every factor of $w$ of length $R$ contains an
  occurrence of $u$.

  Let $A$ and $B$ be two alphabets. A \emph{morphism} from $A^*$ to $B^+$ is a mapping $\varphi: A^* \to B^+$ such that
  $\varphi(uv) = \varphi(u)\varphi(v)$ for all words $u, v \in A^*$. A morphism $\varphi: A^* \to A^+$ is said to be
  prolongable on the letter $a$ if $\varphi(a)$ begins with letter $a$ and $|\varphi(a)| > 1$. Then clearly
  $\varphi^n(a)$ is a prefix of $\varphi^{n+1}(a)$, so we obtain a fixed point
  $\varphi^\omega(a) := \lim_{n \to \infty} \varphi^n(a)$. A morphism $\varphi: A^* \to A^+$ is \emph{primitive} if
  there exists $n \geq 1$ such that for all $a \in A$ the image $\varphi^n(a)$ contains every letter of $A$ at least
  once.

  The word $\del_{i,j}(u)$, where $i + j \leq |u|$, is obtained from the word $u$ by deleting $i$ letters from the
  beginning, and $j$ letters from the end. Let $\varphi$ be a morphism with fixed point $w = \varphi^\omega(a)$. We
  say that a factor $u$ of $w$ admits an \emph{interpretation} $s = (x_0 x_1 \cdots x_{n+1}, i, j)$ by $\varphi$ if
  $u = \del_{i,j}(\varphi(x_0 x_1 \cdots x_{n+1}))$ where $x_i$ are letters, $0 \leq i < |\varphi(x_0)|$ and
  $0 \leq j < |\varphi(x_{n+1})|$. The word $x_0 x_1 \cdots x_{n+1}$ is called the ancestor of the interpretation $s$.
  In this paper we are focused on the Thue-Morse morphism, and in this particular case all sufficiently long factors of
  the fixed point have a unique interpretation (this is demonstrated in \autoref{sec:thue-morse}). Thus it is
  convenient to just talk about the interpretation and the ancestor of $u$ by $\varphi$. In a factor $u$ of $w$ we
  often separate images of letters by bars. For example if $\varphi: 0 \mapsto 01, 1 \mapsto 10$ is the Thue-Morse
  morphism, then the word $01100$ has ancestor $010$, and we place bars as follows: $01|10|0$. If a factor has a unique
  interpretation, then there is only one way to place the bars in that factor, and conversely.

  The \emph{reversal} $\mirror{w}$ of $w = a_1 a_2 \cdots a_n$ is the word $\mirror{w} = a_n \cdots a_2 a_1$; notation
  $w^\sim$ is also used for $\mirror{w}$. If $\mirror{w} = w$, then we say that $w$ is a \emph{palindrome}. By
  convention the empty word is a palindrome. The set of palindromes of $w$ is denoted by $\Pal(w)$. Moreover we define
  $\Pal_w(n) = \Pal(w) \cap \Fac_w(n)$. We say that a word $w$ is \emph{closed under reversal} if for each
  $u \in \Fac(w)$ it holds that $\mirror{u} \in \Fac(w)$. It is well-known that a finite word $w$ contains at most
  $|w| + 1$ distinct palindromic factors (see Proposition 2 of
  \cite{2001:episturmian_words_and_some_constructions_of_de,}). Equality is attained if each position of $w$ introduces
  exactly one new palindrome. \emph{The (palindromic) defect} $\D(w)$ of $w$ is the number $|w| + 1 - |\Pal(w)|$. The
  defect measures the abundance of palindromes in $w$. If $\D(w) = 0$ (i.e., it contains the maximum possible number of
  distinct palindromes), then $w$ is called \emph{rich}. The defect for infinite word $w$ is defined as
  $\D(w) = \sup\{\D(u) : u \text{ is a prefix of } w\}$. The defect of a finite word $w$ is equal to the number of
  prefixes $u$ of $w$ such that $u$ does not have a unioccurrent longest palindromic suffix. For a good reference on
  defect and rich words see \cite{2009:palindromic_richness}.
 
  \emph{Privileged words} were first defined in \cite{2013:a_characterization_of_subshifts_with_bounded_powers} and
  further developed in \cite{2013:introducing_privileged_words_privileged}. The set of privileged words over the
  alphabet $A$, denoted by $\Pri(A)$, is defined inductively as follows:
  \renewcommand{\labelitemi}{$\hyphen$}
  \begin{itemize}
    \item $u \in \Pri(A)$ if $u$ is a complete first return to a shorter privileged word $v \in \Pri(A)$,
    \item the shortest privileged words are the letters of $A$ and the empty word $\varepsilon$.
  \end{itemize}
  The set of privileged factors of a word $w$ is denoted $\Pri(w)$. We define $\Pri_w(n) = \Pri(w) \cap \Fac_w(n)$.
  We set $\A_w(n) = |\Pri_w(n)|$. The function $\A_w$ is called the \emph{privileged complexity function} of $w$.
  Next we list elementary properties of privileged words proved in \cite{2013:introducing_privileged_words_privileged}.
  For completeness we provide full proofs.

  \begin{lemma}\label{lem:privileged_prefix_suffix}
    Let $w$ be a privileged word and $u$ an arbitrary privileged prefix (respectively suffix) of $w$. Then $u$ is a
    suffix (respectively prefix) of $w$.
  \end{lemma}
  \begin{proof}
    If $|w| \leq 1$ or $u = w$, then the claim is clear. Suppose that $|w| \geq 2$ and $|u| < |w|$. By definition $w$ is a
    complete first return to a shorter privileged word $v$. If $|v| < |u|$, then by induction $v$ is a suffix of $u$, and
    thus $v$ would have at least three occurrences in $w$ which is impossible. If $u = v$, then the claim is clear. Finally
    assume that $|v| > |u|$. Then by induction $u$ is a suffix of $v$, and thus a suffix of $w$. The proof in the case that
    the roles of prefix and suffix are reversed is symmetric.
  \end{proof}

  \begin{lemma}
    Let $w$ be a privileged word, and $u$ its longest proper privileged prefix (or suffix). Then $w$ is a complete first
    return to $u$.
  \end{lemma}
  \begin{proof}
    If $|w| \leq 1$, then there is nothing to prove. Suppose that $|w| \geq 2$ and that $w$ is a complete first return to
    privileged word $v$. Now if $|u| > |v|$, then $v$ is a prefix of $u$, and thus by \autoref{lem:privileged_prefix_suffix}
    also a suffix of $u$. Hence $w$ has at least three occurrences of $v$, a contradiction. Therefore $|u| \leq |v|$, and
    by the maximality of $u$, $u = v$, which proves the claim. The proof in the case that the roles of prefix and suffix
    are reversed is symmetric.
  \end{proof}

  \begin{lemma}
    Every position in a word introduces exactly one new privileged factor.
  \end{lemma}
  \begin{proof}
    It is sufficient to show that appending a new letter $a$ to a word $w$ introduces exactly one new privileged
    factor. Since letters are privileged, the word $wa$ has as a suffix a privileged word $u$ of maximal length. Assume
    on the contrary that $u$ occurs in $w$. Then $wa$ has as a suffix a complete first return to $u$, denote it by $v$.
    By definition $v$ is privileged. This contradicts the maximality of $u$, so indeed $u$ does not occur in $w$.
    Finally if appending $a$ introduced another privileged factor, say $z$, then by maximality of $u$, we would have
    that $|z| < |u|$. Thus $z$ would be a suffix of $u$, and by \autoref{lem:privileged_prefix_suffix} $z$ would be a
    prefix of $u$. Consequently $z$ would occur already in $w$ contradicting the assumption that appending $a$
    introduced $z$.
  \end{proof}

  \section{The Privileged Complexity of the Thue-Morse Word}\label{sec:thue-morse}
  In this section we prove a recursive formula for the privileged complexity function of the Thue-Morse word.
  Moreover we study the asymptotic behavior of the function and the occurrences of zeros in its values.

  Let $t = 0110100110010110 \cdots$ be the infinite Thue-Morse word (see Chapter 2 of
  \cite{1983:combinatorics_on_words}). The word $t$ is a fixed point of the morphism $\varphi$ and its square
  $\theta = \varphi^2$.
  \begin{align*}
    \varphi:
    \begin{array}{l}
      0 \mapsto 01 \\
      1 \mapsto 10
    \end{array}
    \qquad
    \theta:
    \begin{array}{l}
      0 \mapsto 0110 \\
      1 \mapsto 1001
    \end{array}
  \end{align*}
  The word $t$ has the following well-known property (\emph{an overlap} is a factor of the form $auaua$ where $a$ is a letter):

  \begin{theorem}
    The Thue-Morse word $t$ does not contain overlaps, that is, it is overlap-free.
  \end{theorem}
  \begin{proof}
    For a proof see e.g. Theorem 2.2.3 of \cite{1983:combinatorics_on_words}.
  \end{proof}

  By this theorem the longest privileged proper border $u$ of a privileged factor $w$ of $t$ cannot overlap with
  itself in $w$, that is, $|u| \leq |w|/2$. In what follows, we implicitly assume this fact. Using overlap-freeness, by
  mere inspection we obtain the following:

  \begin{lemma}\label{lem:unique_interpretation}
    Every factor of $t$ of length at least $4$ admits a unique interpretation by $\varphi$. Every factor of $t$ of length at
    least $7$ admits a unique interpretation by $\theta$. \qed
  \end{lemma}

  Using this lemma we are able to prove the following important proposition:

  \begin{proposition}\label{prp:theta_number_preserving}
    Let $w, u \in \Fac(t)$ be such that $|w| \geq |u| \geq 2$. Then $|\theta(w)|_{\theta(u)} = |w|_u$.
  \end{proposition}
  \begin{proof}
    Clearly always $|\theta(w)|_{\theta(u)} \geq |w|_u$. Say $\theta(u)$ occurs in $\theta(w)$, so
    $\theta(w) = \alpha \theta(u) \beta$ for some words $\alpha$ and $\beta$. There must exist words $\lambda$ and
    $\mu$ such that $\theta(\lambda) = \alpha$ and $\theta(\mu) = \beta$, since otherwise $\theta(u)$ would admit two
    interpretations by $\theta$, which is impossible by \autoref{lem:unique_interpretation} as $|\theta(u)| \geq 8$.
    Hence $w = \lambda u \mu$. This proves that $|\theta(w)|_{\theta(u)} \leq |w|_u$.
  \end{proof}

  The following interesting result can be inferred from Theorem 4.3. in
  \cite{2007:return_words_in_fixed_points_of_substitutions}.
  
  \begin{proposition}\label{prp:tm_number_of_returns}\emph{\cite{2007:return_words_in_fixed_points_of_substitutions}}.
    Every factor except $0$ and $1$ in the Thue-Morse word has exactly $4$ complete returns.
  \end{proposition}

  Let $E$ be the morphism defined by $E(0) = 1$ and $E(1) = 0$.
  As for every $w \in \Fac(t)$ also $E(w) \in \Fac(t)$, we can
  focus primarily on factors beginning with $0$. We let $\Pri_u(n)$ denote the set $\Pri_t(n) \cap u\cdot\{0,1\}^*$,
  and we set $\A_u(n) = |\Pri_u(n)|$. As $111$ is not a factor of $t$, the complete first returns to $0$ are $00, 010$
  and $0110$. Clearly privileged factors beginning with letter $0$ with length greater than one can be divided into
  three groups depending on the first four letters of the word. We have that
  \begin{align*}
    \Pri_0(n) = \Pri_{00}(n) \cup \Pri_{010}(n) \cup \Pri_{0110}(n)
  \end{align*}
  for $n > 1$. Thus for the privileged complexity function $\A_t$ of the Thue-Morse word we have that
  \begin{align*}
    \frac{1}{2} \A_t(n) = \A_{00}(n) + \A_{010}(n) + \A_{0110}(n)
  \end{align*}
  for $n > 1$. Using overlap-freeness, we can easily see that $\Pri_0(1) = \{0\}, \Pri_0(2) = \{00\}, 
  \Pri_0(3) = \{010\}$ and $\Pri_0(4) = \{0110\}$. Hence $\A_t(1) = \A_t(2) = \A_t(3) = \A_t(4) = 2$. Next we state the
  main results of this paper.

  \begin{theorem}\label{thm:tm_privileged_complexity}
    The privileged complexity function $\A_t$ of the Thue-Morse word satisfies
    \begin{align*}
      \A_t(0) & = 1, \A_t(1) = \A_t(2) = \A_t(3) = \A_t(4) = 2, \\
      \frac{1}{2} \A_t(4n) & = 3\A_{00}(n) + \A_{010}(n) + \A_{010}(n+1) + \A_{0110}(n+1) & \text{ for } n \geq 2, \\
      \frac{1}{2} \A_t(4n - 2) & = \A_{00}(4(n-1)) + \A_{010}(4n) + \A_{0110}(4n) & \text{ for } n \geq 2, \\
      \A_t(2n + 1) & = 0 & \text{ for } n \geq 2.
    \end{align*}
  \end{theorem}

  \autoref{thm:tm_privileged_complexity} is directly implied by \autoref{prp:no_odd_length} and the following theorem
  which allows the computation of values of $\A_t$.

  \begin{theorem}\label{thm:collected_formulas}
    The functions $\A_{00}(n), \A_{010}(n)$ and $\A_{0110}(n)$ satisfy
    \begin{align*}
      \A_{00}(4n)       & = 2 \A_{00}(n), \\
      \A_{00}(4n - 2)   & = \A_{0110}(4n), \\
      \A_{0110}(4n)     & = \A_{00}(n) + \A_{010}(n), \\
      \A_{0110}(4n - 2) & = \A_{00}(4(n - 1)), \\
      \A_{010}(4n)      & = \A_{010}(n + 1) + \A_{0110}(n + 1), \\
      \A_{010}(4n - 2)  & = \A_{010}(4n)
    \end{align*}
    for all $n \geq 2$.
  \end{theorem}
  \begin{proof}
    The claim follows from Corollaries \ref{cor:00_formula}, \ref{cor:010_formula} and \ref{cor:0110_formula} proven
    below.
  \end{proof}

  \begin{table}
    \centering
    \begin{tabular}{|c|c|c|c|c|c|c|c|}
    \hline
    2-16 & 18-32 & 34-48 & 50-64 & 66-80 & 82-96 & 98-112 & 114-128\\ \hline
    2    & 2     & 14    & 0     & 2     & 0     & 16     & 0\\
    2    & 2     & 6     & 0     & 2     & 0     & 8      & 0\\
    4    & 4     & 4     & 0     & 2     & 4     & 4      & 6\\
    8    & 8     & 8     & 0     & 2     & 12    & 4      & 18\\
    8    & 8     & 8     & 0     & 2     & 12    & 4      & 18\\
    4    & 4     & 4     & 0     & 2     & 4     & 4      & 6\\
    0    & 6     & 2     & 0     & 2     & 4     & 4      & 8\\
    0    & 14    & 2     & 0     & 2     & 12    & 4      & 24\\
    \hline
    \end{tabular}
    \caption{Values $\A_t(n)$ for $n = 2,4,6,8,\ldots,128$ (the even numbers from $2$ to $128$).}
    \label{tbl:values}
  \end{table}

  Using the formulas of Theorems \ref{thm:tm_privileged_complexity} and \ref{thm:collected_formulas} values of $\A_t$
  can be computed. See \autoref{tbl:values}.

  It is interesting to compare the privileged complexity with the palindromic complexity:

  \begin{theorem}\label{thm:tm_palindromic_complexity}\emph{\cite{2003:palindrome_complexity,2008:palindromic_lacunas_of_the_thue-morse_word}}
    The palindromic complexity function $\PP_t(n)$ of the Thue-Morse word satisfies
    \begin{align*}
      \PP_t(0)      & = 1, \PP_t(1) = \PP_t(2) = \PP_t(3) = \PP_t(4) = 2, \\
      \PP_t(4n)     & = \PP_t(4n - 2) = \PP_t(n) + \PP_t(n + 1) & \text{ for } n \geq 2, \\
      \PP_t(2n + 1) & = 0                                       & \text{ for } n \geq 2.
    \end{align*}
  \end{theorem}

  \begin{table}
    \centering
    \begin{tabular}{|l|l|l|}
      \hline
      $\alpha_1 = 00101100$   & $\beta_1 = 01011010$     & $\gamma_1 = 01100110$ \\
      $\alpha_2 = 00110100$   & $\beta_2 = 010110011010$ & $\gamma_2 = 011010010110$ \\
      $\alpha_3 = 001100$     & $\beta_3 = 010010$       & $\gamma_3 = 0110010110$ \\
      $\alpha_4 = 0010110100$ & $\beta_4 = 0100110010$   & $\gamma_4 = 0110100110$ \\
      \hline
    \end{tabular}
    \caption{The set $\R$ of all complete first returns to $00, 010$ and $0110$ in $t$.}
    \label{tbl:returns}
  \end{table}

  In \autoref{tbl:returns} we list the set $\R$ of all complete first returns to $00, 010$ and $0110$. These words are
  needed later on. We leave it to the reader to verify that these words actually are factors of $t$. By
  \autoref{prp:tm_number_of_returns} the list is exhaustive (this fact is also easily verified directly).

  \begin{lemma}\label{lem:r_prefix}
    If $w \in \Pri_0(n)$ with $n > 4$, then $w$ begins with a word in $\R$.
  \end{lemma}
  \begin{proof}
    Since $|w| > 4$, the word $w$ has proper prefix $u$ where $u$ is one of the words $00, 010$ or $0110$. Since $w$ is
    privileged, the word $u$ is also a suffix of $w$, so $w$ must have a complete first return to $u$ as a prefix. The
    claim follows.
  \end{proof}

  We see that we have at least two privileged factors beginning with 0 of odd length, namely $0$ and $010$. It
  turns out that there are no more:

  \begin{proposition}\label{prp:no_odd_length}\emph{\cite{2013:introducing_privileged_words_privileged}}
    $\A_t(2n + 1) = 0$ for $n \geq 2$.
  \end{proposition}
  \begin{proof}
    We may focus on privileged factors beginning with $0$. Let $w, |w| > 4,$ be a privileged factor of $t$ beginning
    with $0$. Now $w$ begins with one of the three privileged words $00$, $010$ or $0110$. With respect to the morphism
    $\varphi$ the bars must be placed as follows: $0|0$, $01|0$ or $0|10$ and $01|10$. Now if $w$ begins with $00$
    (respectively $0110$), then it also ends with $00$ (respectively $0110$), and by the placement of the bars we
    immediately see that $w$ has even length. Assume then that $w$ begins with $010$. As $|w| > 4$, by
    \autoref{lem:r_prefix} $w$ has as a prefix one of the words
    $\beta_1 = 01|01|10|10, \beta_2 = 01|01|10|01|10|10, \beta_3 = 0|10|01|0$ or $\beta_4 = 0|10|01|10|01|0$ (bars with
    respect to $\varphi$). If $w$ begins with some $\beta_i$, then it also ends with $\beta_i$. From the placement of
    the bars we see that $|w|$ is necessarily even.
  \end{proof}

  We frequently need complete information on short privileged factors. The next lemma provides this knowledge. In what
  follows, we assume the lemma to be known.

  \begin{lemma}\label{lem:short_privileged}
    Let $w \in \Pri_0(n)$ with $n \leq 12$. Then $w \in \{0,010,0110\} \cup \R$.
  \end{lemma}
  \begin{proof}
    It was already remarked that if $|w| \leq 4$, then $w \in \{0,010,0110\}$. If $|w| > 4$, then by
    \autoref{lem:r_prefix} the word $w$ begins with a word in $\R$. If $w \notin \R$, then as $w$ is privileged, $w$
    must have as a prefix a complete first return to some word in $\R$. This prefix $u$ cannot overlap with itself.
    Since $|w| \leq 12$, it thus follows that $|u| \leq 6$. The words of minimal length in $\R$ have length $6$, so
    $w = u^2$ and $u \in \{\alpha_3, \beta_3\}$. However, both $\alpha_3^2$ and $\beta_3^2$ contain a third power (that
    is, an overlap) yielding a contradiction. Thus $w \in \R$.
  \end{proof}

  Next we characterize the different classes of privileged factors in the Thue-Morse word. In what follows, we say that
  a word $w$ is \emph{$\theta$-invertible} if there exists a word $u$ such that $\theta(u) = w$. Recall the words
  $\alpha_i, \beta_i$ and $\gamma_i$ from \autoref{tbl:returns}.

  \begin{lemma}\label{lem:00_characterization}
    Let $w \in \Pri_{00}(n)$ for some $n > 2$. Then
    \begin{enumerate}[(i)]
      \item $4 \mid |w| \Longleftrightarrow 1w110$ or $011w1$ is a $\theta$-invertible factor of $t$ $\Longleftrightarrow
            w$ begins with $\alpha_1$ or $\alpha_2$,
      \item $4 \nmid |w| \Longleftrightarrow 1w1$ is a $\theta$-invertible factor of $t$ $\Longleftrightarrow
            w$ begins with $\alpha_3$ or $\alpha_4$.
    \end{enumerate}
  \end{lemma}
  \begin{proof}
    By \autoref{lem:unique_interpretation} all factors of $t$ of length at least $7$ admit a unique interpretation by
    $\theta$, so in $\alpha_1, \alpha_2$ and $\alpha_4$ there is a unique way to place bars:
    $\alpha_1 = 001|0110|0, \alpha_2 = 0|0110|100$ and $\alpha_4 = 001|0110|100$. For $\alpha_3$ there are potentially
    two ways to place bars: $\alpha_3 = 001|100$ and $\alpha_3 = 0|0110|0$. However the latter is not possible as
    $(0110)^3$ is not a factor of $t$.

    \emph{(i)} Assume that $4 \mid |w|$. If $w$ begins with $\alpha_4$, then it also ends with $\alpha_4$.
    From the placement of the bars it can be seen that this is not possible; it would follow that $4 \nmid |w|$.
    Similarly $w$ cannot begin with $\alpha_3$. Hence $w$ must begin with $\alpha_1$ or $\alpha_2$.
    On the other hand if $w$ begins with $\alpha_1$ or $\alpha_2$, then
    $1w110$ or $011w1$ must be $\theta$-invertible by the placement of the bars. Then clearly $4 \mid |w|$.

    \emph{(ii)} Assume that $4 \nmid |w|$. By (i) $w$ has to begin with $\alpha_3$ or $\alpha_4$. In either case
    $1w1$ is a $\theta$-invertible factor of $t$. The other direction is also clear: if $w$ begins with $\alpha_3$ or
    $\alpha_4$, then by (i) it must be that $4 \nmid |w|$.
  \end{proof}

  \begin{lemma}\label{lem:010_characterization}
    Let $w \in \Pri_{010}(n)$ for some $n \geq 1$. Then
    \begin{enumerate}[(i)]
      \item $4 \mid |w| \Longleftrightarrow 10w01$ is a $\theta$-invertible factor of $t$ $\Longleftrightarrow
            w$ begins with $\beta_1$ or $\beta_2$,
      \item $4 \nmid |w|, 2 \mid |w| \Longleftrightarrow 011w110$ is a $\theta$-invertible factor of $t$ $\Longleftrightarrow
            w$ begins with $\beta_3$ or $\beta_4$,
      \item $4 \nmid |w|, 2 \nmid |w| \Longleftrightarrow w = 010$.
    \end{enumerate}
  \end{lemma}
  \begin{proof}
    From \autoref{prp:no_odd_length} it follows that (iii) holds.

    As in the previous proof, we know the placements of bars in the words $\beta_1, \beta_2, \beta_3$ and $\beta_4$:
    $\beta_1 = 01|0110|10, \beta_2 = 01|0110|0110|10, \beta_3 = 0|1001|0$ and $\beta_4 = 0|1001|1001|0$.

    \emph{(i)} Assume that $4 \mid |w|$. As in the previous proof, from the placement of the bars we see that $w$
    cannot begin with $\beta_3$ or $\beta_4$, and hence it must start with $\beta_1$ or $\beta_2$. Thus $10w01$ is
    $\theta$-invertible. Again the unique placement of the bars implies that the converse is also true.

    \emph{(ii)} By (i) it is enough to note that if $w$ begins with $\beta_3$ or $\beta_4$, then $011w110$ is
    $\theta$-invertible.
  \end{proof}

  \begin{lemma}\label{lem:0110_characterization}
    Let $w \in \Pri_{0110}(n)$ for some $n > 4$. Then
    \begin{enumerate}[(i)]
      \item $4 \mid |w| \Longleftrightarrow w$ is a $\theta$-invertible factor of $t$ $\Longleftrightarrow
            w$ begins with $\gamma_1$ or $\gamma_2$,
      \item $4 \nmid |w| \Longleftrightarrow 10w$ or $w01$ is a $\theta$-invertible factor of $t$ $\Longleftrightarrow
            w$ begins with $\gamma_3$ or $\gamma_4$.
    \end{enumerate}
  \end{lemma}
  \begin{proof}
    The placement of bars is known: $\gamma_1 = 0110|0110, \gamma_2 = 0110|1001|0110,
    \gamma_3 = 01|1001|0110$ and $\gamma_4 = 0110|1001|10$. As in the two previous proofs, by looking at the
    placements of the bars the claim straightforwardly follows.
  \end{proof}

  The three previous lemmas allow us to prove the following useful result:

  \begin{corollary}\label{cor:combined_characterization}
    Let $w \in \Pri_t(n)$ and $u$ be its longest privileged proper prefix such that $|u| > 4$. Then $4 \mid |w|$ if and only if
    $4 \mid |u|$.
  \end{corollary}
  \begin{proof}
    Let $\mathcal{S} = \{\alpha_1, \alpha_2, \beta_1, \beta_2, \gamma_1, \gamma_2\}$. Assume that $4 \mid |w|$. The
    lemmas \ref{lem:00_characterization}, \ref{lem:010_characterization} and \ref{lem:0110_characterization} imply that
    $w$ begins with some $v \in \mathcal{S}$. Since $|u| > 4$, it follows from \autoref{lem:r_prefix} that $u$ has as a
    prefix a word in $\R$. Since no word in the set $\R$ is a proper prefix of any word in $\R$, it follows that $u$
    begins with $v$. The same three lemmas now imply that $4 \mid |u|$. On the other hand if $4 \mid |u|$, then $u$
    begins with a word in $\mathcal{S}$. Consequently $w$ begins with a word in $\mathcal{S}$, so $4 \mid |w|$.
  \end{proof}

  Let us then define the following functions $f_1, g_1: \{0,1\}^* \to \{0,1\}^* :
  f_1(w) = \del_{1,3}(\theta(1w))$ and $g_1(w) = \del_{3,1}(\theta(w1))$.

  \begin{lemma}\label{lem:00_reduction}
    Let $n \geq 2$. The function $f_1$ is a bijection $\Pri_{00}(n) \to \Pri_{\alpha_1}(4n)$ and the function $g_1$ is
    a bijection $\Pri_{00}(n) \to \Pri_{\alpha_2}(4n)$.
  \end{lemma}
  \begin{proof}
    We will first prove the claim for $f_1$. If $n = 2$, then $\Pri_{00}(2) = \{00\}$ and
    $\Pri_{\alpha_1}(8) = \{\alpha_1\}$, so the claim indeed holds. The latter part of this proof shows that if
    $\Pri_{\alpha_1}(4n) \neq \emptyset,$ then also $\Pri_{00}(n) \neq \emptyset$. Thus the claim holds also if
    $n = 3, 4, 5$, as then $\Pri_{00}(n) = \emptyset$. Assume that $n > 5$.

    Let $w \in \Pri_{00}(n)$, and let $v$ be its longest privileged proper prefix. Note that now $|v| \geq 2$. As $v$
    begins with $00$ it follows by induction that $f_1(v) \in \Pri_{\alpha_1}(t)$. As $v$ is always preceded by the
    letter $1$, $w = vw'1v$, and thus
    \begin{align*}
      f_1(w) = \del_{1,3}(\theta(1v)\theta(w')\theta(1v)) = \del_{1,3}(\theta(1v))110 \theta(w') 1 \del_{1,3}(\theta(1v)) = f_1(v) 110 \theta(w') 1 f_1(v).
    \end{align*}
    By \autoref{lem:00_characterization} the factor $f_1(v)$ is always preceded by $1$ and followed by $110$. Thus
    if $f_1(v)$ would occur more than twice in $f_1(w)$, then $\theta(v) = 1 f_1(v) 110$ would occur more than twice in 
    $\theta(w)$ which is impossible by \autoref{prp:theta_number_preserving}. Hence $f_1(w)$ is a complete first
    return to the privileged word $f_1(v)$, and thus $f_1(w) \in \Pri_{\alpha_1}(4n)$.

    Assume then that $w \in \Pri_{\alpha_1}(4n)$. By \autoref{lem:00_characterization} there exists
    $z \in \Fac_t(n + 1)$ such that $\theta(z) = 1w110$. Write $u = \del_{1,0}(z)$. Then $f_1(u) = w$. Let $v$ be the
    longest privileged proper prefix of $w$. The assumption $n > 5$ implies that $|v| > 4$ (the maximum length of a
    word in $\R$ is $12$), so by \autoref{cor:combined_characterization} we have that $4 \mid |v|$.
    Thus by induction there exists $s \in \Pri_{00}(t)$ such that $f_1(s) = v$. Thus
    $f_1(u) = w = f_1(s) \cdots f_1(s)$, and so $u$ begins and ends with $s$. Now if $s$ would occur more than twice
    in $u$, then as $s$ is always preceded by $1$, $f_1(s) = v$ would occur more than twice in $w$ which is impossible.
    Thus $u$ is a complete first return to $s$, so $u \in \Pri_{00}(n)$.

    Now the claim for the function $g_1$ follows as $f_1(w)^\sim = g_1(\mirror{w})$ and $\mirror{\alpha}_1 = \alpha_2$.
  \end{proof}

  \begin{lemma}\label{lem:00_reduction_4n-2}
    Let $n \geq 1$. The function $f_2: \Pri_{00}(4n -2) \to \Pri_{1001}(4n), f_2(w) = 1w1$ is a bijection.
  \end{lemma}
  \begin{proof}
    If $n = 1$, then $\Pri_{00}(2) = \{00\}$ and $\Pri_{1001}(4) = \{1001\}$, so the claim holds. The cases
    $n = 2,3$ are also clear: $\Pri_{00}(6) = \{\alpha_3\}, \Pri_{1001}(8) = \{E(\gamma_1)\},
    \Pri_{00}(10) = \{\alpha_4\}$ and $\Pri_{1001}(12) = \{E(\gamma_2)\}$; see \autoref{lem:short_privileged}.
    Assume that $n > 3$.
    
    Let $w \in \Pri_{00}(4n - 2)$. As the factor $00$ is always preceded and followed by the letter $1$, $1w1 \in
    \Fac(t)$, and $1w1$ begins and ends with $1001$. Let $v$ be the longest privileged proper prefix of $w$. The
    assumption $n > 3$ implies that $|v| > 4$, so by \autoref{cor:combined_characterization} it holds that
    $4 \nmid |v|$. Thus by induction the word $1v1$ is privileged. The word $1w1$ is a complete first return to $1v1$,
    as otherwise $w$ would contain more than two occurrences of $v$. Thus $1w1 \in \Pri_{1001}(4n)$.

    Let then $1w1 \in \Pri_{1001}(4n)$. Again by applying \autoref{cor:combined_characterization} and induction to the
    longest privileged proper prefix of $1w1$ we get that $w \in \Pri_{00}(4n - 2)$.
  \end{proof}

  \begin{corollary}\label{cor:00_formula}
    $\A_{00}(4n) = 2 \A_{00}(n)$ and $\A_{00}(4n - 2) = \A_{0110}(4n)$ for all $n \geq 2$.
  \end{corollary}
  \begin{proof}
    As the ranges of the functions $f_1: \Pri_{00}(n) \to \Pri_{\alpha_1}(4n)$ and
    $g_n: \Pri_{00}(n) \to \Pri_{\alpha_2}$ are disjoint, the claim follows since by \autoref{lem:00_characterization}
    $\Pri_{00}(4n) = \Pri_{\alpha_1}(4n) \cup \Pri_{\alpha_2}(4n)$. The other equality follows directly from
    \autoref{lem:00_reduction_4n-2} as $\A_{1001}(m) = \A_{0110}(m)$ for all $m \geq 4$.
  \end{proof}

  \begin{lemma}\label{lem:010_reduction}
    Let $n \geq 2$. The function $f_3: \Pri_{101}(n + 1) \cup \Pri_{1001}(n + 1) \to \Pri_{010}(4n), f_3(w) = \del_{2,2}(\theta(w))$ is a bijection.
  \end{lemma}
  \begin{proof}
    If $n = 2$, then $\Pri_{101}(3) = \{101\}, \Pri_{1001}(3) = \emptyset$ and $\Pri_{010}(8) = \{\beta_1\}$.
    If $n = 3$, then $\Pri_{101}(4) = \emptyset, \Pri_{1001}(4) = \{1001\}$ and $\Pri_{0110}(12) = \{\beta_2\}$. We
    may assume that $n > 3$.

    Let $w \in \Pri_{101}(n + 1) \cup \Pri_{1001}(n + 1)$, and $v$ its longest privileged proper prefix. By induction
    $f_3(v) \in \Pri_{010}(t)$. As $v$ is a prefix and a suffix of $w$, $f_3(w)$ starts and ends with $f_3(v)$. By
    \autoref{lem:010_characterization} the word $f_3(v)$ is always preceded by $10$ and followed by $01$. Thus if
    $f_3(w)$ contained more than two occurrences of $f_3(v)$, then \autoref{prp:theta_number_preserving} would imply
    that $w$ contains more than two occurrences of $v$ which would be a contradiction. We conclude that
    $f_3(w) \in \Pri_{010}(4n)$.
    
    Let then $w \in \Pri_{010}(4n)$. By \autoref{lem:010_characterization} there is such a word $u$ that $f_3(u) = w$.
    Let $v$ be the longest privileged proper prefix of $w$. By the assumption $n > 3$, we have that $|v| > 4$. By
    \autoref{cor:combined_characterization} we can apply induction to obtain a word
    $s \in \Pri_{101}(t) \cup \Pri_{1001}(t)$ such that $f_3(s) = v$. Therefore $f_3(u) = w = f_3(s) \cdots f_3(s)$.
    By \autoref{lem:010_characterization} $f_3(s)$ is always preceded by $10$ and followed by $01$. Therefore $u$
    begins and ends with $s$. Now if $u$ would contain a third occurrence of $s$, then $w$ would contain a third
    occurrence of $v$, which is not possible. Hence $u \in \Pri_{101}(n + 1) \cup \Pri_{1001}(n + 1)$.
  \end{proof}

  \begin{lemma}\label{lem:010_reduction_4n-2}
    Let $n \geq 2$. The function $f_4: \Pri_{101}(4n - 2) \to \Pri_{010}(4n), f_4(w) = 0w0$ is a bijection.
  \end{lemma}
  \begin{proof}
    If $n = 2$, then $\Pri_{101}(6) = \{E(\beta_3)\}$ and $\Pri_{010}(8) = \{\beta_1\}$. If $n = 3$, then
    $\Pri_{101}(10) = \{E(\beta_4)\}$, and $\Pri_{010}(12) = \{\beta_2\}$. Thus it can be assumed that $n > 3$.

    Let $w \in \Pri_{101}(4n - 2)$, and $v$ its longest privileged proper prefix. As $n > 3$, we have that $|v| > 4$.
    By \autoref{cor:combined_characterization} and induction $f_4(v) \in \Pri_{010}(t)$. By
    \autoref{lem:010_characterization} the factor $f_4(v)$ is always preceded and followed by letter $0$. Thus it can
    be written that $f_4(w) = f_4(v) \cdots f_4(v)$. If there was a third occurrence of $f_4(v)$ in $f_4(w)$, then in
    $w$ there would be at least three occurrences of $v$, which is false. Therefore $f_4(w) \in \Pri_{010}(4n)$.

    Let $0w0 \in \Pri_{010}(4n)$. Again, by applying \autoref{cor:combined_characterization} and induction to the
    longest privileged proper prefix of $0w0$, we get that $w \in \Pri_{101}(4n - 2)$.
  \end{proof}

  \begin{corollary}\label{cor:010_formula}
    $\A_{010}(4n) = \A_{010}(n + 1) + \A_{0110}(n + 1)$ and $\A_{010}(4n - 2) = \A_{010}(4n)$ for all $n \geq 2$.
  \end{corollary}
  \begin{proof}
    This follows directly from Lemmas \ref{lem:010_reduction} and \ref{lem:010_reduction_4n-2} as
    $\A_{101}(n) = \A_{010}(n)$ and $\A_{1001}(n) = \A_{0110}(n)$ for all $n \geq 0$.
  \end{proof}

  \begin{lemma}\label{lem:0110_reduction}
    Let $n \geq 2$. The function $\theta: \Pri_{00}(n) \cup \Pri_{010}(n) \to \Pri_{0110}(4n)$ is a bijection.
  \end{lemma}
  \begin{proof}
    If $n = 2$, then $\Pri_{00}(2) = \{00\}, \Pri_{010}(2) = \emptyset$ and $\Pri_{0110}(8) = \{\gamma_1\}$. If
    $n = 3$, then $\Pri_{00}(3) = \emptyset, \Pri_{010}(3) = \{010\}$ and $\Pri_{0110}(12) = \{\gamma_2\}$. By the
    argument at the end of the proof if $\Pri_{0110}(4n) \neq \emptyset$, then
    $\Pri_{00}(n) \cup \Pri_{010}(n) \neq \emptyset$. As $\Pri_{00}(n) \cup \Pri_{010}(n) = \emptyset$ when $n = 4$,
    it can be further assumed that $n > 4$.

    Let $w \in \Pri_{00}(n) \cup \Pri_{010}(n)$, and $v$ its longest privileged proper prefix. Now $|v| \geq 2$ as
    $n > 4$. Once again by induction $\theta(v) \in \Pri_{0110}(t)$. By \autoref{prp:theta_number_preserving} the word
    $\theta(w)$ must be a complete first return to $\theta(v)$, that is, $\theta(w) \in \Pri_{0110}(4n)$.

    Let $w \in \Pri_{0110}(4n)$. By \autoref{lem:0110_characterization} there is $u \in \Fac_t(n)$ such that
    $\theta(u) = w$. Again by \autoref{cor:combined_characterization} and induction there exists
    $s \in \Pri_{00}(t) \cup \Pri_{010}(t)$ such that $\theta(s) = v$ where $v$ is the longest privileged proper prefix
    of $w$. It follows that $u$ must be a complete first return to $s$, so $u \in \Pri_{00}(n) \cup \Pri_{010}(n)$.
  \end{proof}

  \begin{lemma}\label{lem:0110_reduction_4n-2}
    Let $n \geq 2$. The function $f_4: \Pri_{11}(4n) \to \Pri_{0110}(4n + 2), f_4(w) = 0w0$ is a bijection.
  \end{lemma}
  \begin{proof}
    If $n = 2$, then $\Pri_{11}(8) = \{E(\alpha_1), E(\alpha_2)\}$ and $\Pri_{0110}(10) = \{\gamma_3, \gamma_4\}$.
    If $n = 3$, then $\Pri_{11}(12) = \emptyset$ and $\Pri_{0110}(14) = \emptyset$. The latter set is empty because if
    $w \in \Pri_{0110}(14)$, then $w$ has as a prefix and a suffix a word $u$ which is a complete first return to
    $0110$. Since $|u| \geq 8$, these two occurrences of $u$ in $w$ overlap. This is contradictory.

    The rest of the proof is by induction along the lines of the proof of \autoref{lem:00_reduction_4n-2}.
  \end{proof}

  \begin{corollary}\label{cor:0110_formula}
    $\A_{0110}(4n) = \A_{00}(n) + \A_{010}(n)$ and $\A_{0110}(4n - 2) = \A_{00}(4(n - 1))$ for $n \geq 2$.
  \end{corollary}
  \begin{proof}
    The result follows from Lemmas \ref{lem:0110_reduction} and \ref{lem:0110_reduction_4n-2} as $\A_{00}(n) = \A_{11}(n)$
    for all $n \geq 0$.
  \end{proof}

  \autoref{prp:no_odd_length} and Corollaries \ref{cor:00_formula}, \ref{cor:010_formula} and \ref{cor:0110_formula}
  together prove \autoref{thm:tm_privileged_complexity}. In the following we characterize the non-primitive privileged
  factors:

  \begin{proposition}\label{prp:primitive_privileged}
    The only non-primitive privileged factors of $t$ beginning with $0$ are $00, \beta_3, \gamma_1$ and $\gamma_2^2$.
  \end{proposition}
  \begin{proof}
    Let $w$ be a non-primitive privileged factor of $t$ beginning with $0$. Since $t$ is overlap-free, it cannot
    contain third powers. Thus $w = u^2$ for some privileged factor $u$. If $|u| = 1$, then $w = 00$. If $|u| > 1$,
    then $u$ cannot begin with $00$ as otherwise $w$ would have $0^4$ as a central factor. Hence $u$ begins with $010$
    or $0110$. If $|u| = 3,4$, then $w = \beta_3, \gamma_1$. We may assume that $|u| > 5$. Then as $|u|$ is even, we
    have that $4 \mid |w|$, so $4 \mid |u|$ by \autoref{cor:combined_characterization}. By
    \autoref{lem:010_characterization} if $u$ begins with $010$, then $u$ begins with $\beta_1$ or $\beta_2$, and so
    $w$ has $\beta_1^2$ or $\beta_2^2$ as a central factor. This is, however, impossible as neither $\beta_1^2$ nor
    $\beta_2^2$ is a factor of $t$. Thus $u$ must begin with $0110$, and begin with $\gamma_1$ or $\gamma_2$. The word
    $\gamma_1$ is non-primitive, and thus $\gamma_1^2 \notin \Fac(t)$. As $\gamma_2$ is primitive, and $\gamma_2^2$ is
    not, we may assume that $u$ begins with $\gamma_2$ and that $u \neq \gamma_2$. Thus $w$ has
    $\gamma_2^2 = \theta(010)^2$ as a central factor. By \autoref{lem:0110_reduction} $w = \theta(v)$ where $v$ is a
    privileged word beginning with $010$. As $(010)^2$ must be preceded by $011$ and followed by $110$, we have that
    $w$ has $\theta(011010 \cdot 010110) = \theta^2(010)$ as a central factor. As $010$ cannot be preceded and followed
    by the same letter, we have that $v$ ends with $\beta_1$ and begins with $\beta_2$, or symmetrically $v$ ends with
    $\beta_2$ and begins with $\beta_1$. Either case is impossible as $v$ is privileged.
  \end{proof}

  Next we study the asymptotic behavior of the function $\A_t$.

  \begin{proposition}\label{prp:privileged_unbounded}
    $\underset{n \to \infty}\limsup \A_t(n) = \infty$ and $\underset{n \to \infty}{\liminf \vphantom{p}} \A_t(n) = 0$.
  \end{proposition}
  \begin{proof}
    The fact that the inferior limit is $0$ already follows from \autoref{prp:no_odd_length}. We will prove that
    when $n \geq 6$,
    \begin{align*}
      \A_t(2^n) =
      \begin{cases}
        3\cdot2^{\frac{1}{2}(n - 1)}, &\text{ if } n \text{ is odd,}\\
        0, &\text{ if } n \text{ is even,}
      \end{cases}
    \end{align*}
    which proves the claim.

    As $2^{n-2} + 1$ is odd, by \autoref{cor:010_formula} it holds
    for $n > 3$ that
    \begin{align*}
      \A_{010}(2^n) = \A_{010}(2^{n-2}+1) + \A_{0110}(2^{n-2} + 1) = 0.
    \end{align*}
    Thus by \autoref{thm:tm_privileged_complexity}
    \begin{align*}
      \frac{1}{2}\A_t(2^n) = 3\A_{00}(2^{n-2}).
    \end{align*}
    Now $\A_{00}(2) = 1$ and $\A_{00}(4) = 0$. By \autoref{cor:00_formula} $\A_{00}(2^n) = 2\A_{00}(2^{n-2})$, so
    for all $n \geq 1$,
    \begin{align*}
      \A_{00}(2^n) =
      \begin{cases}
        2^{\frac{1}{2}(n - 1)}, &\text{ if } n \text{ is odd},\\
        0,                      &\text{ if } n \text{ is even},
      \end{cases}
    \end{align*}
    proving the desired equality.
  \end{proof}

  This result is very interesting. It can be proven that the palindromic complexity function of a fixed point of a
  primitive morphism is bounded (see
  \cite{2000:palindrome_complexity_bounds_for_primitive_substitution,2003:palindrome_complexity}). So as the Thue-Morse
  word is a fixed point of a primitive morphism, the above results demonstrate that in general the palindromic
  complexity function and the privileged complexity function of an infinite word can behave radically differently.

  Our next aim is to show that there exist arbitrarily long gaps of zeros in the values of $\A_t$. For this we need
  several lemmas.
  
  Let us define an integer sequence $(a_n)$ as follows: $a_1 = 14$ and $a_n = 4(a_{n-1} - 2) + 2(-1)^n$ for
  $n > 1$.  The first few terms of the sequence are $14, 50, 190, 754, 3006, \ldots$ Note that $a_n$ is always even,
  and not divisible by $4$.

  \begin{lemma}\label{lem:gaps_left_end}
    If $n$ is even, then $\A_{00}(a_n - 2) = \A_{010}(a_n - 2) = 0$ and $\A_{0110}(a_n - 2) = 1$.
    If $n > 1$ is odd, then $\A_{00}(a_n - 2) = \A_{0110}(a_n - 2) = 0$ and $\A_{010}(a_n - 2) = 1$. Moreover if
    $n > 1$, then $\A_t(a_n - 2) = 2$.
  \end{lemma}
  \begin{proof}
    Using the formulas of Corollaries \ref{cor:00_formula}, \ref{cor:010_formula} and \ref{cor:0110_formula} it is readily
    verified that the claim holds for $n = 2$.

    Let $n > 1$ be odd. Then $a_n = 4(a_{n-1} - 2) - 2$. Using induction, \autoref{prp:no_odd_length} and the formulas
    of Corollaries \ref{cor:00_formula}, \ref{cor:010_formula} and \ref{cor:0110_formula} we get that
    \begin{align*}
      \A_{00}(a_n - 2)   &= 2\A_{00}(a_{n-1} - 3) = 0,\\
      \A_{010}(a_n - 2)  &= \A_{010}(a_{n-1} - 2) + \A_{0110}(a_{n-1} - 2) = \A_{0110}(a_{n-1} - 2) = 1 \text{ and }\\
      \A_{0110}(a_n - 2) &= \A_{00}(a_{n-1} - 3) + \A_{010}(a_{n-1} - 3) = 0.
    \end{align*}

    Let $n$ be even, so $a_n = 4(a_{n-1} - 2) + 2$. Similarly as above
    \begin{align*}
      \A_{00}(a_n - 2)   &= 2\A_{00}(a_{n-1} - 2) = 0, \\
      \A_{010}(a_n - 2)  &= \A_{010}(a_{n-1} - 1) + \A_{0110}(a_{n-1} - 1) = 0 \text{ and }\\
      \A_{0110}(a_n - 2) &= \A_{00}(a_{n-1} - 2) + \A_{010}(a_{n-1} - 2) = \A_{010}(a_{n-1} - 2) = 1.
    \end{align*}
    
    It clearly follows that $\A_t(a_n - 2) = 2$ for $n > 1$.
  \end{proof}

  By \autoref{lem:short_privileged} it can be verified that $\A_t(12) = 4$, so $\A_t(a_n - 2) \neq 0$ for all $n \geq 1$.

  %\begin{lemma}\label{lem:gaps_left_end_zero}
  %  $\A_t(a_n) = 0$ for all $n \geq 1$.
  %\end{lemma}
  %\begin{proof}
  %  We will prove that $\A_t(a_n + 2) = 0$ for all $n \geq 1$. The claim follows then from the formula
  %  \begin{align*}
  %    \frac{1}{2} \A_t(4n - 2) & = \A_{00}(4(n-1)) + \A_{010}(4n) + \A_{0110}(4n)
  %  \end{align*}
  %  of Theorem \ref{thm:tm_privileged_complexity} together with Lemma \ref{lem:gaps_left_end}.
  %
  %  Using the Theorem it can be verified that $\A_t(14) = \A_t(16) = \A_t(50) = \A_t(52) = 0$, so the claim
  %  holds when $n = 1, 2$. Let $n$ be even, so $a_n + 2= 4(a_{n-1} - 1)$. By Theorem \ref{thm:tm_privileged_complexity}
  %  \begin{align*}
  %    \frac{1}{2} \A_t(a_n + 2) & = 3\A_{00}(a_{n-1} - 1) + \A_{010}(a_{n-1} - 1) + \A_{010}(a_{n-1}) + \A_{0110}(a_{n-1})\\
  %                              & = \A_{010}(a_{n-1}) + \A_{0110}(a_{n-1})\\
  %                              & = 0,
  %  \end{align*}
  %  where the last equality follows from induction.

  %  Let then $n$ be odd, so that $a_n + 2 = 4a_{n-1}$. As above
  %  \begin{align*}
  %    \frac{1}{2} \A_t(a_n + 2) & = 3\A_{00}(a_{n-1}) + \A_{010}(a_{n-1}) + \A_{010}(a_{n-1} + 1) + \A_{0110}(a_{n-1} + 1)\\
  %                              & = 3\A_{00}(a_{n-1}) + \A_{010}(a_{n-1})\\
  %                              & = 0.
  %  \end{align*}
  %\end{proof}

  \begin{lemma}\label{lem:gaps_right_end}
    $\A_t(2^n + 2) = 2$ for all even $n \neq 2$.
  \end{lemma}
  \begin{proof}
    We will prove that $\A_{00}(2^n + 2) = \A_{0110}(2^n + 2) = 0$ and $\A_{010}(2^n + 2) = 1$. The claim follows from
    this. This claim is readily verified when $n = 0$. Let then $n > 2$ be even. By the formulas of Corollaries
    \ref{cor:00_formula} and \ref{cor:0110_formula} we obtain that
    \begin{align*}
      \A_{00}(2^n + 2) = \A_{0110}(2^n + 4) = \A_{00}(2^{n-2} + 1) + \A_{010}(2^{n-2} + 1) = 0
    \end{align*}
    as $2^{n-2} + 1$ is odd and greater than $3$ when $n > 2$ is even. Similarly using in addition
    \autoref{cor:010_formula} we get by applying induction that
    \begin{align*}
      \A_{010}(2^n + 2)  &= \A_{010}(2^n + 4) = \A_{010}(2^{n-2} + 2) + \A_{0110}(2^{n-2} + 2) = 1 \text{ and } \\
      \A_{0110}(2^n + 2) &= \A_{00}(2^n) = 0,
    \end{align*}
    where the last equality was proven already in the proof of \autoref{prp:privileged_unbounded}.
  \end{proof}

  Finally we have proven enough lemmas in order to prove the following result:

  \begin{proposition}\label{prp:gaps_formula}
    For all $n \geq 1$ if $a_n - 1 \leq k \leq 2^{2(n + 1)} + 1$, then $\A_t(k) = 0$. Also for all $n \geq 1$
    $\A_t(a_n - 2) \neq 0$ and $\A_t(2^{2(n+1)} + 2) \neq 0$.
  \end{proposition}
  \begin{proof}
    By inspection it can be verified that indeed if $a_1 - 1 = 13 \leq k \leq 17 = 2^{2(1 + 1))} + 1$, then $\A_t(k) = 0$.
    Let $n > 1$. Assume that $a_n \leq k \leq 2^{2(n + 1)}$. If $k$ is odd, then $\A_t(k) = 0$. Suppose that $k$ is even.
    Assume that $4$ divides $k$. Then it follows that $a_{n-1} - 1 \leq k/4 \leq 2^{2n}$, so by the induction hypothesis
    $\A_t(k/4) = 0$. By \autoref{thm:tm_privileged_complexity}
    \begin{align*}
      \frac{1}{2} \A_t(k) & = 3\A_{00}(k/4) + \A_{010}(k/4) + \A_{010}(k/4+1) + \A_{0110}(k/4+1) = 0.
    \end{align*}
    Assume that $4$ does not divide $k$. Now $a_{n-1} - 1 \leq \frac{k + 2}{4} \leq 2^{2n}$ and $a_{n-1} - 2 \leq \frac{k - 2}{4} \leq 2^{2n}$.
    Now using the already familiar formulas we get
    \begin{align*}
      \frac{1}{2}\A_t(k) &= \A_{00}(k - 2) + \A_{010}(k + 2) + \A_{0110}(k + 2) \\
                         &= 2 \A_{00}\left(\frac{k - 2}{4}\right) + \A_{010}\left(\frac{k+2}{4} + 1\right) + \A_{0110}\left(\frac{k+2}{4} + 1\right) +
                            \A_{00}\left(\frac{k+2}{4}\right) + \A_{010}\left(\frac{k+2}{4}\right) \\
                         &= 2\A_{00}\left(\frac{k-2}{4}\right),
    \end{align*}
    where the last equality follows from the induction hypothesis. Now if $k \neq a_n$, then $a_{n - 1} - 1 \leq \frac{k-2}{4}$, so
    by the induction hypothesis $\A_t(k) = 0$. If $k = a_n$, then $\frac{k-2}{4} = a_{n-1} - 2$. Then however by \autoref{lem:gaps_left_end}
    $\A_{00}(a_{n-1} - 2) = 0$, so also $\A_t(k) = 0$.
    
    The claim now follows as $a_n - 1$ and $2^{2(n + 1)} + 1$ are odd. Earlier it was proved that
    $\A_t(a_n - 2), \A_t(2^{2(n+1)} + 2) \neq 0$.
  \end{proof}

  Straightforwardly using induction it can be proved that for $n \geq 3$ it holds that $a_n < 2^{2n + 1} + 2^{2n} < 2^{2(n + 1)}$, so
  in particular if for $n \geq 3$ it holds that $2^{2n + 1} + 2^{2n} \leq k \leq 2^{2(n + 1)}$, then $\A_t(k) = 0$. This verifies the
  following result which was conjectured in \cite{2013:introducing_privileged_words_privileged}:

  \begin{corollary}\label{cor:zero_gaps}
    There exists arbitrarily long (but not infinite) gaps of zeros in the values of the privileged complexity function of the Thue-Morse word. \qed
  \end{corollary}

  \autoref{cor:zero_gaps} raises a natural question: If the privileged complexity function of a word $w$ contains arbitrarily
  large gaps of zeros, does it follow that $\limsup_{n \to \infty} \A_w(n) = \infty$? It is conceivable that the large
  gaps force large values of $\A_w(n)$ between the gaps. On the other hand the gaps could occur so sparsely that
  $\A_w(n)$ is still bounded. The author was not able to answer this question.

  The privileged complexity function of the Thue-Morse word is complicated. Even though the Thue-Morse morphism has
  really nice properties, finding the recursive formula for the function is a long task. On the other hand without the
  nice properties of the morphism, the work may not have been possible at all. Indeed if the morphism were not uniform,
  then it would have been harder to calculate the length of the privileged factors. Other crucial property of the
  morphism is its \emph{circularity}: every image of a letter is uniquely determined by its first or last letter. The
  author thinks that it could be possible to obtain results on the privileged complexity of fixed points of primitive
  uniform circular morphisms other than Thue-Morse.

  \section{Privileged Palindrome Complexity}\label{sec:privileged_palindrome}
  In this section we will focus on the privileged palindrome complexity function of the Thue-Morse word $t$. This
  function $\B_t$ counts the number of $n$-length factors of $t$ which are both privileged and palindromic. The
  arguments given are similar to those of \autoref{sec:thue-morse}.
  
  Similar to \autoref{sec:thue-morse} we write $\M_u(n) = \Pri_t(n) \cap \Pal_t(n) \cap u\cdot\{0,1\}^*$,
  and $\B_u(n) = |\M_u(n)|$. Again it suffices to consider factors beginning with letter $0$.

  \begin{lemma}\label{lem:palindrome_matches}
    Let $w \in \Pal(t)$. Then $w \in \Pal_t(4n)$ for some $n \geq 1$ if and only if $w$ is $\varphi$-invertible (i.e.,
    there exists a word $u \in \Fac(t)$ such that $\varphi(u) = w$).
  \end{lemma}
  \begin{proof}
    The claim holds for all palindromes of length less than or equal to $4$, they are: $0,1,00,11,010,101,0110$ and
    $1001$. Suppose that $n \geq 2$.

    Let $w \in \Pal_t(4n)$ be a shortest palindrome that is not $\varphi$-invertible. Suppose first that $w$ begins with
    $00$, that is, $w = 001w'100$. Now $1w'1 \in \Pal(4(n-1))$, so by the minimality of $|w|$ the word $1w'1$ is
    $\varphi$-invertible.  As $w$ begins with $00$, the word $1w'1$ must have two interpretations by $\varphi$. This is
    a contradiction with \autoref{lem:unique_interpretation}.  Say $w$ begins with $01$ (the case where it begins with
    $10$ is symmetric), so it can be written that $w = 01w'10$. As $w$ is not $\varphi$-invertible neither is $w'$
    (otherwise $w$ would have two interpretations by $\varphi$) which is a contradiction with the minimality of $|w|$.

    Suppose then that $w$ is the shortest $\varphi$-invertible palindrome such that $4 \nmid |w|$. It can be written
    that $w = 01w'10$ (or symmetrically $w = 10w'01$), so $w'$ is a palindrome of length $|w| - 4$ which is
    $\varphi$-invertible. This contradicts the choice of $w$.
  \end{proof}

  Let us have a closer look at the function $\theta$ and the functions
  \begin{align*}
    f_2&: w \mapsto 1w1 \\
    f_3&: w \mapsto \del_{2,2}(\theta(w)) \\
    f_4&: w \mapsto 0w0
  \end{align*}
  defined in Lemmas \ref{lem:00_reduction_4n-2}, \ref{lem:010_reduction} and \ref{lem:010_reduction_4n-2}.  Obviously
  all functions preserve palindromes. By Lemmas \ref{lem:00_reduction_4n-2}, \ref{lem:010_reduction},
  \ref{lem:010_reduction_4n-2}, \ref{lem:0110_reduction} and \ref{lem:0110_reduction_4n-2} the functions
  $f_2, f_3, f_4$ and $\theta$ preserve privileged words. Hence we have that the following functions are bijections:
  \begin{align*}
    f_2&: \M_{00}(4n-2) \to \M_{1001}(4n), w \mapsto 1w1, \\
    f_3&: \M_{101}(n+1) \cup \M_{1001}(n+1) \to \M_{010}(4n), w \mapsto \del_{2,2}(\theta(w)), \\
    f_4&: \M_{101}(4n-2) \to \M_{010}(4n), w \mapsto 0w0, \\
    f_4&: \M_{11}(4n) \to \M_{0110}(4n + 2), w \mapsto 0w0, \\
    \theta&: \M_{00}(n) \cup \M_{010}(n) \to \M_{0110}(4n), w \mapsto \theta(w).
  \end{align*}
  We have thus proved the following formulas:
  \begin{align*}
    \B_{00}(4n-2)  &= \B_{0110}(4n), \\
    \B_{010}(4n-2) &= \B_{010}(4n), \\
    \B_{0110}(4n-2)&= \B_{00}(4(n-1)), \\
    \B_{010}(4n)   &= \B_{010}(n+1) + \B_{0110}(n+1), \\
    \B_{0110}(4n)  &= \B_{00}(n) + \B_{010}(n),
  \end{align*}
  for $n \geq 2$. We are still missing a formula for $\B_{00}(4n)$. However $\M_{00}(4n) = \emptyset$ by
  \autoref{lem:palindrome_matches}, so $\B_{00}(4n) = 0$.  By putting together these formulas we get the following
  result:

  \begin{theorem}\label{thm:palandrome_formula}
    The privileged palindrome complexity function $\B_t$ of the Thue-Morse word satisfies
    \begin{align*}
      \B_t(0) & = 1, \B_t(1) = \B_t(2) = \B_t(3) = \B_t(4) = 2, \\
      \frac{1}{2} \B_t(4n) & = \B_{00}(n) + \B_{010}(n) + \B_{010}(n+1) + \B_{0110}(n+1) & \text{ for } n \geq 2, \\
      \B_t(4n - 2) & = \B_t(4n) & \text{ for } n \geq 2, \\
      \B_t(2n + 1) & = 0 & \text{ for } n \geq 2.
    \end{align*} \qed
  \end{theorem}

  As in the previous section for the function $\A_t$, we study next the asymptotic behavior of the function $\B_t$ and
  the occurrences of zeros in its values.

  Let us define an integer sequence $(b_n)$ as follows: $b_1 = 6$ and $b_n = 4b_{n-1} - 2$ for $n > 1$. The first few
  terms of the sequence are $6, 22, 86, 342, 1366, \ldots$ Note that $b_n$ is always even, and not divisible by $4$.

  \begin{lemma}\label{lem:4_sequence}
    $\B_t(b_n) = 4$ for all $n \geq 1$.
  \end{lemma}
  \begin{proof}
    By inspection $\B_{00}(6) = 1, \B_{010}(6) = 1$ and $\B_{0110}(6) = 0$, so $\B(6) = 4$. We will prove that
    $\B_{00}(b_n) = 2$ and $\B_{010}(b_n) = \B_{0110}(b_n) = 0$ for all $n > 1$. The claim follows from this. Now
    \begin{align*}
      \B_{00}(b_n)   &= \B_{0110}(b_n + 2) = \B_{00}(b_{n-1}) + \B_{010}(b_{n-1}) \\
      \B_{010}(b_n)  &= \B_{010}(b_n + 2) = \B_{010}(b_{n-1} + 1) + \B_{0110}(b_{n-1} + 1), \\
      \B_{0110}(b_n) &= \B_{00}(b_n - 2) = \B_{00}(b_{n-1} - 1).
    \end{align*}
    So the claim is indeed true.
  \end{proof}

  \begin{proposition}\label{prp:palandrome_values}
    The function $\B_t$ takes values in $\{0,1,2,4\}$, and the values $0,2$ and $4$ are attained infinitely often.
  \end{proposition}
  \begin{proof}
    As $\B_t(0) = 1, \B_t(1) = \B_t(2) = \B_t(3) = \B_t(4) = 2$, by \autoref{thm:palandrome_formula} we need only to
    consider the values $\B_t(4n)$. If $n = 2$, then $\B_t(8) = 4$. Let $n > 4$ be even. Then by
    \autoref{thm:palandrome_formula}
    \begin{align*}
      \frac{1}{2} \B_t(4n) = \B_{00}(n) + \B_{010}(n).
    \end{align*}
    It suffices to prove that if $\B_{00}(n) \neq 0$ then $\B_{010}(n) = 0$. We are only interested in the case
    where $n$ is not divisible by $4$, as otherwise $\B_{00}(n) = 0$. Now
    \begin{align*}
      \B_{00}(n)  &= \B_{0110}(n + 2) = \B_{00}\left(\frac{n+2}{4}\right) + \B_{010}\left(\frac{n+2}{4}\right) \text{ and } \\
      \B_{010}(n) &= \B_{010}(n+2) = \B_{010}\left(\frac{n+2}{4} + 1\right) + \B_{0110}\left(\frac{n+2}{4} + 1\right).
    \end{align*}
    Clearly if $\B_{00}(n) \neq 0$, then $(n+2)/4$ is even, and thus $\B_{010}(n) = 0$.

    Let $n$ be odd. Then
    \begin{align*}
      \frac{1}{2} \B_t(4n) = \B_{010}(n + 1) + \B_{0110}(n + 1).
    \end{align*}
    Again it suffices to prove that if $\B_{010}(n) \neq 0$, then $\B_{0110}(n) = 0$. As $\B_{0110}(n) = 0$, when
    $n$ is not divisible by $4$, we need to consider only the case where $n$ is divisible by $4$. Now
    \begin{align*}
      \B_{010}(n)  &= \B_{010}\left(\frac{n}{4} + 1\right) + \B_{0110}\left(\frac{n}{4} + 1\right) \text{ and } \\
      \B_{0110}(n) &= \B_{00}\left(\frac{n}{4}\right) + \B_{010}\left(\frac{n}{4}\right).
    \end{align*}
    Obviously if $\B_{010}(n) \neq 0$, then $n/4$ is odd, and thus $\B_{0110}(n) = 0$.

    By \autoref{lem:4_sequence} the function $\B_t$ takes value $4$ infinitely often. Moreover the arguments of Lemmas
    \ref{lem:gaps_left_end} and \ref{lem:gaps_right_end} work if the function $\A_t$ is replaced with the function
    $\B_t$. Thus the value $2$ is also attained infinitely often.
  \end{proof}

  Let us now consider the gaps of zeros in the values of $\B_t$ like we did for the function $\A_t$ in
  \autoref{prp:gaps_formula}. It is clear by \autoref{prp:gaps_formula} that if $a_n - 1 \leq k \leq 2^{2(n+1)} + 1$,
  then $\B_t(k) = 0$. The arguments of Lemmas \ref{lem:gaps_left_end} and \ref{lem:gaps_right_end} work if the function
  $\A_t$ is replaced with the function $\B_t$ (in the proof of the latter lemma, the last equality follows now more
  easily as $\B_{00}(4n) = 0$ for all $n \geq 1$), so $\B_t(a_n - 2), \B_t(2^{2(n+1)}+2) \neq 0$ for all $n \geq 1$.
  Therefore the function $\B_t$ has the same gaps as the function $\A_t$ described by \autoref{prp:gaps_formula}, the
  gaps do not widen.

  \section{A Comparison of Palindromes and Privileged Words}\label{sec:comparison}
  In this section we compare the behavior of palindromes and privileged words and the behavior of the respective
  complexity functions in general.

  Let $w$ be an infinite word. In \cite{2013:introducing_privileged_words_privileged} and
  \cite{2013:a_characterization_of_subshifts_with_bounded_powers} the following relation between the sets $\Pal(w)$ and
  $\Pri(w)$ was proved:

  \begin{proposition}\label{prp:rich_privileged=palindrome}\emph{\cite{2013:introducing_privileged_words_privileged,2013:a_characterization_of_subshifts_with_bounded_powers}}
    A finite or infinite word $w$ is rich if and only if $\Pri(w) = \Pal(w)$.
  \end{proposition}

  First we strengthen this result a bit. Now if a word is rich, then from the previous proposition it obviously follows
  that the palindromic and privileged complexity functions of $w$ coincide. Next we prove the surprising fact that the
  converse is also true. We start with a lemma which is interesting in its own right.

  \begin{lemma}\label{lem:defect_shortest_inclusion}
    Let $w$ be a finite or infinite word. If $w$ is not rich, then there exists a shortest privileged factor $u$ which
    is not a palindrome. Moreover $\Pal_w(n) = \Pri_w(n)$ for all $n$ such that $0 \leq n < |u|$ and
    $\Pal_w(|u|) \subsetneq \Pri_w(|u|)$.
  \end{lemma}
  \begin{proof}
    If $w$ is not rich, then there exists a position $n$ such that no new palindrome in position $n$ is introduced.
    However position $n$ introduces a new privileged factor, which thus cannot be a palindrome. Hence there exists a
    shortest privileged factor $u$ which is not a palindrome. By the minimality of $|u|$ it follows that
    $\Pri_w(n) \subseteq \Pal_w(n)$ for all $n$ such that $0 \leq n < |u|$. Let $p, |p| > 1,$ be a minimal length
    palindrome which is not privileged. Let $q$ be the longest proper palindromic suffix of $p$. By minimality $q$ is
    privileged. As $p$ is not privileged, it has as a suffix a complete first return to $q$, say $v$. As $q$ is the
    longest palindromic suffix of $p$, $v$ is not a palindrome. By minimality of $|u|$ we have that
    $|p| > |v| \geq |u|$, so $\Pal_w(n) \subseteq \Pri_w(n)$ for all $n$ such that $0 \leq n \leq |u|$.
  \end{proof}

  \begin{theorem}
    A finite or infinite word $w$ is rich if and only if $\PP_w(n) = \A_w(n)$ for all $n$ such that $0 \leq n \leq |w|$.
  \end{theorem}
  \begin{proof}
    The fact that the condition is necessary follows from \autoref{prp:rich_privileged=palindrome}. Assume that
    $\PP_w(n) = \A_w(n)$ for all $n$ such that $0 \leq n \leq |w|$. If $w$ is not rich, then by
    \autoref{lem:defect_shortest_inclusion} there exists such $n$ that $\Pal_w(n) \subsetneq \Pri_w(n)$, so
    $\PP_w(n) < \A_w(n)$, which is a contradiction. Therefore $w$ is rich.
  \end{proof}

  For instance, the Thue-Morse word has, as factors, the word $00101100$ which is privileged and not palindromic, and
  the palindrome $00101100110100$ which is not privileged. Thus for a word $w$ its possible that neither of the sets
  $\Pal(w)$ and $\Pri(w)$ is included in the other. By \autoref{lem:defect_shortest_inclusion} it follows that if
  $\Pri(w) \subseteq \Pal(w)$, then $\Pri(w) = \Pal(w)$, i.e., $w$ is rich. Next it is natural to ask if there are
  examples of infinite words $w$ such that $\Pal(w)$ is properly contained in $\Pri(w)$. It turns out that this is
  possible, but not in the case of uniformly recurrent words containing infinitely many palindromes. We begin with a
  simple observation.

  \begin{lemma}\label{lem:pal_inclusion_defect}
    Let $w$ be a recurrent infinite word. If $\Pal(w) \subsetneq \Pri(w)$, then $w$ has infinite defect.
  \end{lemma}
  \begin{proof}
    As $\Pal(w) \subsetneq \Pri(w)$ there exists a privileged factor $u$ which is not a palindrome. Consider any factor
    $v$ which is a complete first return to $u$ in $\Fac(w)$ (such a factor exists as $w$ is recurrent). Let $z$ be a
    prefix of $w$ having $v$ as a suffix. Let $p$ be the longest palindromic suffix of $z$. By the assumption
    $\Pal(w) \subsetneq \Pri(w)$ the palindrome $p$ is also privileged. If $|p| > |u|$, then $p$ has $\mirror{u}$ as as
    a prefix. Since $p$ has a privileged suffix $u$, it also has $u$ as a prefix, so $\mirror{u} = u$, which is
    impossible. Thus $|p| < |u|$, so $p$ is actually the longest palindromic suffix of $u$. As $z$ contains two
    occurrences of $u$, it follows that the longest palindromic suffix of $z$ is not unioccurrent. As $w$ is recurrent,
    there are infinitely many prefixes of $w$ having $v$ as a suffix. Thus $w$ has infinite defect.
  \end{proof}

  Next we define an infinite binary word $\kappa = \lim_{n \to \infty} u_n$ as the limit of the sequence
  $u_0 = 00101100$, $u_{n+1} = u_n 0^n u_n$. It is clear that $\kappa$ is recurrent and aperiodic, and contains
  infinitely many palindromes of the form $0^n$. The word $\kappa$ is, however, not closed under reversal as
  $(1011)^\sim = 1101$ is not a factor.  We claim that $\Pal_\kappa(n) = \{0^n, 10^{n-2} 1\}$ for $n \geq 7$. Let
  $p \in \Pal_\kappa(n)$ for $n \geq 7$, and $m$ be minimal such that $p$ occurs in $u_m$. As $u_{m-1}$ starts and ends
  with $00101100$, and $1101$ is not a factor of $\kappa$, we conclude that $p$ must be a central factor of $u_m$.
  There are thus only two possibilities, $0^m$ or $10^{m-2} 1$. By direct inspection the reader can verify that
  $\Pal_w(6) = \{0^6\}, \Pal_w(5) = \{0^5\}, \Pal_w(4) = \{0000, 0110\}$ and $\Pal_w(3) = \{000,010,101\}$. Thus we
  have proved the following:

  \begin{lemma}\label{lem:recurrent_pal_inclusion}
    There exists an infinite recurrent aperiodic binary word $w$ having the following properties: $w$ is not closed
    under reversal, $w$ contains infinitely many palindromes and $\Pal(w) \subsetneq \Pri(w)$. \qed
  \end{lemma}

  Now let us consider the \emph{Chacon word} $\lambda$, the fixed point of the (non-primitive) morphism
  $0 \mapsto 0010, 1 \mapsto 1$ \cite{1995:les_transformations_de_chacon_combinatoire}. The word $\lambda$ is
  aperiodic, and uniformly recurrent, as the letter $0$ occurs in bounded gaps. By a direct verification one can show
  that the word $\lambda$ does not contain palindromes of length $13$ or $14$. Therefore $\Pal_\lambda(n) = \emptyset$
  for all $n \geq 13$. There are total $23$ palindromes in $\lambda$. Using the same brute-force approach one can show
  that all palindromes in $\lambda$ are privileged. The Chacon word is not closed under reversal: for instance
  $(100100)^\sim = 001001 \notin \Fac(\lambda)$ as $001001$ cannot be properly factored over the set $\{0010, 1\}$. We
  have:

  \begin{lemma}\label{lem:unif_recurrent_not_closed_pal_inclusion}
    There exists an infinite uniformly recurrent aperiodic binary word $w$ having the following properties: $w$ is not
    closed under reversal, $w$ contains finitely many palindromes and $\Pal(w) \subsetneq \Pri(w)$. \qed
  \end{lemma}

  Next we recall a construction of \cite{2009:infinite_words_without_palindrome}. Consider the infinite word
  $\mu = \lim_{n \to \infty} u_n$, the limit of the sequence $u_0 = 01, u_{n+1} = u_n 23 \mirror{u}_n$. The word
  $\mu$ is uniformly recurrent, aperiodic, closed under reversal and contains only finitely many palindromes, namely
  only the letters $0,1,2$ and $3$. By applying the morphism
  \begin{align*}
    h:
    \begin{array}{l}
      0 \mapsto 101, \\
      1 \mapsto 1001, \\
      2 \mapsto 10001, \\
      3 \mapsto 100001 \\
    \end{array}
  \end{align*}
  to the word $\mu$ the authors obtain a uniformly recurrent aperiodic binary word which is closed under reversal and
  contains only finitely many palindromes (the longest is of length $12$). By direct inspection it can be verified that
  each palindrome in $h(\mu)$ is privileged. Hence we have:

  \begin{lemma}\label{lem:unif_recurrent_closed_pal_inclusion}
    There exists an infinite uniformly recurrent aperiodic binary word $w$ having the following properties: $w$ is
    closed under reversal, $w$ contains finitely many palindromes and $\Pal(w) \subsetneq \Pri(w)$. \qed
  \end{lemma}

  However it turns out that if a uniformly recurrent word contains infinitely many palindromes, then the inclusion can
  not be proper. Note that such a word is necessarily closed under reversal.

  \begin{proposition}\label{prp:unif_recurrent_inf_pal_rich}
    Let $w$ be a uniformly recurrent word containing infinitely many palindromes. If $\Pal(w) \subseteq \Pri(w)$, then
    $\Pal(w) = \Pri(w)$, that is, $w$ is rich.
  \end{proposition}
  \begin{proof}
    Assume on the contrary that $\Pal(w) \subseteq \Pri(w)$ and that $w$ is not rich. Then there exists a privileged
    factor $u$ which is not a palindrome. Since $w$ is uniformly recurrent $u$ is a factor of some palindrome $p$.
    Clearly $u$ cannot be a central factor of $p$. Thus there exists a central factor $q$ of $p$ which begins with
    $u$ and ends with $\mirror{u}$ (or $q$ begins with $\mirror{u}$ and ends with $u$, but this case is symmetric).
    It is immediate that $q$ is a palindrome. Thus by the assumption $q$ is privileged. As $q$ has as a prefix the
    privileged word $u$, the prefix $u$ also occurs as a suffix of $q$. Hence $u$ is a palindrome, a contradiction.
  \end{proof}

  Note that in the proof uniform recurrence was only needed to establish that $u$ is a factor of some palindrome.
  Thus to obtain the result it is only necessary to suppose that every privileged factor occurs in some palindrome.

  \section{Acknowledgments}
  The author was supported by a FiDiPro grant (137991) from the Academy of Finland.

  We thank the referees for carefully reading the manuscript. The comments greatly improved the presentation.

  \printbibliography
 \end{document}
% vim: set textwidth=119: